\documentclass[11pt]{amsart}
\usepackage{amsmath,amsfonts,amsthm,amssymb,amscd}

\usepackage{color}
\usepackage{enumitem}

\newtheorem{theorem}{Theorem}[section]
\newtheorem{lemma}[theorem]{Lemma}

\theoremstyle{definition}

\newtheorem{definition}[theorem]{Definition}

\newtheorem{example}[theorem]{Example}

\newcommand{\Chi}{{\raise2pt\hbox{$\chi$}}}
\newcommand{\ssm}{\smallsetminus}

\newcommand{\ra}{\rightarrow}

\newcommand{\N}{\mathbb{N}}

\newcommand{\Z}{\mathbb{Z}}

\newcommand{\fN}{{\mathfrak N}}

\DeclareMathOperator{\Min}{Min}

\DeclareMathOperator{\Sing}{Sing}

\DeclareMathOperator{\reg}{reg}

\DeclareMathOperator{\supp}{supp}

\newcommand\medcap{{\,\raise1.5pt\hbox{{$\scriptstyle{\bigcap}$}}\,}}
\newcommand\medcup{{\,\raise1.5pt\hbox{{$\scriptstyle{\bigcup}$}}\,}}
\begin{document}

\title{Semi-Complemented Commutative Group Rings}

\author[W.Wm. McGovern]{Warren Wm. McGovern}
\address{H. L. Wilkes Honors College, Florida Atlantic University, 5353 Parkside Dr., Jupiter, FL 33458}
\email{warren.mcgovern@fau.edu (corresponding author)}

\author[Y. Zhou]{Yiqiang Zhou}
\address{Department of Mathematics and Statistics, Memorial University of Newfoundland, St.John's, NL A1C 5S7, Canada}
\email{zhou@mun.ca}



\begin{abstract}
Recall that an element $x\in R$ is complemented if there is a $y\in R$ such that $xy=0$ and $x+y\in \reg(R)$. In a recent article \cite{bbmz}, the authors investigated those rings for which every non-nilpotent element is complemented, calling such rings {\it semi-complemented}. As the title of the current work suggests we characterize when a commutative group $RG$ is semi-complemented.

\end{abstract}
\maketitle
\thispagestyle{empty}

\section{Introduction}
Throughout, we assume that all rings are commutative with identity different than 0. We let $\reg(R)$ and $U(R)$ denote the set of regular elements and units, respectively. The nil-radical of $R$ is denoted by $\fN(R)$. Recall the following well-known concept.

\begin{definition}
Let $R$ be a ring. The element $x\in R$ is called {\bf complemented} if there is a $y\in R$ such that $xy=0$ and $x+y\in \reg(R)$. If every element of $R$ is complemented then $R$ is said to be complemented.
\end{definition}

It is a fact that every regular element is complemented and that $0$ is the only nilpotent element that is complemented. This latter fact led the authors of \cite{bbmz} to study those rings for which every non-nilpotent element is complemented, calling such rings {\bf semi-complemented}. Obviously, complemented rings are precisely the reduced semi-complemented rings. Another such class of rings are what the authors termed {\bf Property $D$}; the ring $R$ has Property $D$ if $R=\reg(R)\cup\fN(R)$. (We encourage the interested reader to read \cite{bbmz} about the history of this class of rings including other names for this condition.) We recall the interesting characterization of semi-complemented rings.

\begin{theorem}\cite[Theorem 2.21]{bbmz}
The ring $R$ is semi-complemented if and only if $R$ is complemented or has Property $D$.
\end{theorem}

Our goal here is to characterize the semi-complemented group rings. We do this in two parts. First, we characterize the complemented group rings. Second, characterize those group rings that have Property $D$.
\vspace{.2in}

Throughout, $G$ is assumed to be an abelian group written multiplicatively. We let $t(G)$ denote torsion subgroup of $G$. For $g\in t(G)$, we let $o(g)$ denote the order of $g$, and we denote by $o(G)$ the set of orders of elements in $t(G)$. (Recall that $G$ is said to be torsion-free if $t(G)=\{1_G\}$.)
$R$ denotes a commutative ring, and we use $RG$ to denote the group ring. We shall write elements of $RG$ as $\sum r_gg$ for appropriate $r_g\in R$ and $g\in G$ (with the caveat that $r_g=0$ for all but finitely mane $g\in G$).

For each prime $p\in \N$, we let $G_p=\{g\in G: o(g)=p^n$ for some $n\in\N\}$, the $p$-subgroup of $G$. We define $\supp(G)$ as the set of primes $p$ for which $G_p\neq \{1_G\}$. We let $\Sing_R(G)$ denote the subgroup of $G$ generated by those $G_p$ such that $p\in U(R)$.
For a prime $p\in\N$, we let $\fN_p(R)=\{r\in R: pr\in\fN(R)\}$.

The proofs of the following are known; see \cite{karpilovsky}.

\begin{theorem}
1) $R$ is indecomposable and $\Sing_R(G)=\{e_G\}$  if and only if $RG$ is indecomposable.

2) $RG$ is an integral domain if and only if $R$ is an integral domain and $G$ is torsion-free.

3) $RG$ is reduced if and only if $R$ is reduced and $o(G)\subseteq \reg(R)$.
\end{theorem}

\begin{theorem}\cite[Theorem 4.2]{karpilovsky}
Let $R$ be a commutative ring with identity and $G$ an abelian group.
$$\fN(RG)=\fN(R)G+<r(g-1): r\in \fN_p(R), g\in G_p \hbox{ for some } p\in \supp(G)\}.$$

\end{theorem}

When $G$ is torsion-free, $RG$ has the property that if $\alpha\in RG$ is a zero-divisor, then there is some non-zero $r\in R$, such that $r\alpha =0$. We shall call this property an application of McCoy's Theorem, as the proof follows from McCoy's Theorem on polynomial rings.

\vspace{.2in}

\section{Complemented Group Rings}
The purpose of this section is to classify when a commutative group ring is complemented. Now is a good time to point out that complemented rings are precisely those rings whose classical ring of quotients are von Neumann regular; $q(R)$ denotes the classical ring of quotients of $R$. We give some other useful characterizations. Recall that $\Min(R)$ denotes the space of minimal prime ideals of $R$ equipped with the hull-kernel topology (aka the Zariski topology). For more information as well as proofs of the above see \cite[Theorem 4.5]{huckaba}.

We restate the important classification.

\begin{theorem}\cite[Theorem 4.5]{huckaba}
Let $R$ be a reduced ring. The following are equivalent.
\begin{enumerate}[label={\rm \arabic*.}, nolistsep]
\item $R$ is complemented.
\item $q(R)$ is von Neumann regular.
\item $\Min(R)$ is compact and $R$ satisfies Property $A$.
\item $\Min(R)$ is compact and $R$ satisfies the annihilator condition.
\end{enumerate}
\end{theorem}

\vspace{.2in}

In \cite[Theorem  2.3]{gs}, the authors give sufficient conditions for when $RG$ is complemented. Specifically, the authors use the property that every $n\in o(G)$ is a unit in $R$; {\it $R$ is uniquely divisible by every order of an element in $G$}. However, their proof only uses that each $n\in o(G)$ is regular in $R$ and thus a unit in $q(R)$. The containment $o(G)\subseteq \reg(R)$ means that every $n\in o(G)$ is regular in $R$. Observe that $o(G)\subseteq \reg(R)$ if and only if $\supp(G)\subseteq \reg(R)$. Notice that if $G$ is torsion-free, then $o(G)=\{1\}$. Thus, Theorem 2.3 of \cite{gs} states and proves that if $o(G)\subseteq \reg(R)$ and $R$ is complemented, then $RG$ is complemented.

We prove the following two theorems. Together, these provide a complete classification of when $RG$ is complemented.
\vspace{.2in}

\begin{theorem}\label{Thm1}
If $G$ is a torsion abelian group, then $RG$ is complemented if and only if $R$ is complemented and $o(G)\subseteq \reg(R)$.
\end{theorem}
\vspace{.1in}

\begin{theorem}\label{Thm2}
If $G$ is an abelian group which is not a torsion group, then $RG$ is complemented if and only if $R$ is reduced with compact minimal prime spectrum and $o(G)\subseteq \reg(R)$.
\end{theorem}
\vspace{.1in}

The following lemma will be useful.

\begin{lemma}\label{Lem1}
Let $G=<g>$ be a finite cyclic group of order $n$. If the element $a\in R$ is complemented in $RG$, then the element $a\in R$ is complemented in $R$.
\end{lemma}
\vspace{.1in}

\begin{proof} \underline{{\it Lemma \ref{Lem1}.}}
Let $a\in R$ and let $G=\{e,g,g^2,\ldots, g^{n-1}\}$. By hypothesis, there is some $t\in RG$ which is a complement for $a$, that is $at=0$ and $a+t$ is regular in $RG$. Write
$$t=\sum_{i=0}^{n-1} b_i g^i.$$
Observe that $ab_i=0$ for each $i$. Set $x=b_0+\ldots +b_{n-1}$ and thus $ax=0$. We show that $a+x$ is regular in $R$. Suppose $c\in R$ and $c(a+x)=0$. Note that $ca=-cx$ and since $R$ is reduced $ca=0=cx$. Let
$$s=\sum_{i=0}^{n-1} cg^i.$$
Since $s$ is a multiple of $c$ it follows that $sa=0$. A straightforward calculation shows that
\begin{eqnarray*}
s(a+t)   &=& sa+st \\
   &=&  st \\
   &=& c (\sum_{i=0}^{n-1} xg^i) \\
   &=& 0.
\end{eqnarray*}
Since $a+t$ is regular in $RG$, it follows that $s=0$, whence $c=0$.
\end{proof}
\vspace{.1in}

We can now prove Theorem \ref{Thm1}.
\vspace{.1in}

\begin{proof} \underline{{\it Theorem \ref{Thm1}.}} We suppose that $G$ is a torsion group. If $R$ is complemented and $o(G)\subseteq \reg(R)$, then $RG$ is complemented. The proof of this is given in \cite{gs}. The main point is that if for each finitely generated subgroup $H$ of $G$, $RH$ is complemented, then so is $RG$.

Conversely, suppose that $RG$ is complemented. Then $RG$ is reduced in which case it follows that $o(G)\subseteq \reg(R)$ (\cite[Theorem 5]{connell}). Let $a\in R$. There is some $t\in RG$ such that $at=0$ and $a+t$ is regular in $RG$. Let $H$ be the subgroup of $G$ generated by the support of $t$. Then $t\in RH$ and $a$ is complemented in $RH$. Now, by hypothesis $H$ is torsion. Since it is finitely generated $H$ is a finite abelian group and therefore $H\cong H_1\times \cdots \times H_k$ where each $H_i$ is a finite cyclic group. Then,
$$RH\cong R(H_1\times \cdots \times H_k) \cong (R(H_1\times \cdots \times H_{k-1}))H_k.$$
An application of Lemma \ref{Lem1}, yields that $a$ is complemented in $R(H_1\times \cdots \times H_{k-1})$. Applying Lemma \ref{Lem1} several more times yields that $a\in R$ is complemented in $R$.
\end{proof}

\vspace{.2in}

We prove another lemma.

\begin{lemma}\label{Lem2}
Let $R$ be a reduced and $G$ an abelian group. The embedding of $R$ into $RG$ induces a continuous surjective map $\psi:\Min(RG)\ra \Min(R)$.
\end{lemma}
\vspace{.1in}

\begin{proof} \underline{{\it Lemma \ref{Lem2}.}} We define $\psi:\Min(RG)\ra \Min(R)$ by $\psi(Q)=Q\cap R$. Since the embedding of $R$ into $RG$ is a flat morphism it is known that this map is well-defined (see \cite{picavet}). That the map is continuous with respect to the respective hull-kernel topologies follows from \cite[Proposition 1.11]{bdm}.

That the map is surjective is true for any ring extension (see the exercises of Section 1.6 of \cite{kaplansky}).
\end{proof}
\vspace{.1in}

And now are in position to prove Theorem \ref{Thm2}.
\vspace{.1in}

\begin{proof}  \underline{{\it Theorem \ref{Thm2}.}}
Suppose $G$ is not torsion and let $g\in G$ be a torsion-free element. We first suppose that $R$ is reduced with $\Min(R)$ is compact, and that $o(G)\subseteq \reg(R)$. The authors \cite{gs} take care of the case when $G$ is torsion-free. Their proof can be modified to show that if for each finitely generated subgroup $H$ of $G$ there is another finitely generated subgroup $H\subseteq H'$ of $G$ such that $RH'$ is complemented, then $RG$ is complemented. Thus, if $H$ is a finitely generated subgroup of $G$, expand $H$ to a finitely generated subgroup $H'$ of $G$ which contains a torsion-free element. Then we may write $H'=H_1H_2$ where $H_1$ is a torsion-free group and $H_2$ is a torsion group. Since $RH'\cong (RH_1)H_2$, the hypothesis implies that $RH_1$ is complemented, and subsequently $(RH_1)H_2$ is complemented.

Conversely, suppose that $RG$ is complemented. Then certainly $o(G)\subseteq \reg(R)$ since a complemented ring is reduced. By Lemma \ref{Lem2}, $\Min(R)$ is a continuous image of $\Min(RG)$, the latter of which is compact. Therefore, $\Min(R)$ is compact.
\end{proof}

\begin{example}
Quentel's Example $Q$ is a well-known example of a commutative ring that has compact minimal prime spectrum but is not complemented See \cite[Section 27]{huckaba}). Furthermore, if one so wishes, one can take $Q$ to have characteristic 0, so that $o(G)\subseteq \reg(Q)$ is automatically satisfied and thus for any non-torsion group $G$, $RG$ is complemented.

\end{example}

\vspace{.2in}

\section{Property $D$}

\begin{lemma}\label{fg}
Let $R$ be a commutative ring and $G$ an abelian group. Then $RG$ has Property $D$ if and only if for every finitely generated subgroup $H$ of $G$, $RH$ has Property $D$.
\end{lemma}

\begin{proof} The necessity follows from the fact that every subring of a ring with Property $D$ also has Property $D$. For the sufficiency, let $\alpha \in RG \ssm \fN(RG)$, and suppose that $\alpha \beta = 0$ for $\beta\in RG$. There exists a finitely generated
subgroup $H$ of $G$ such that both $\alpha,\beta\in \in RH$. As $\alpha\notin \fN(RH)$, it follows that $\alpha\in\reg(RH)$, whence $\beta =0$. Thus, $\alpha\in\reg(RG)$.
\end{proof}

\begin{lemma}\label{tf}
Let $R$ be a commutative ring and suppose that $G$ is a torsion-free abelian group.
Then $RG$ has Property $D$ if and only if $R$ has Property $D$.
\end{lemma}

\begin{proof}
If $RG$ has Property $D$, then so does every subring, in particular, so does $R$.
\vspace{.1in}

Suppose that $R$ has Property $D$, and suppose that $\alpha\in RG\ssm \fN(RG)$. Writing $\alpha =\sum_{i=1}^n a_ig_i$, it follows that for some $i=1,\cdots n$, $a_i\notin \fN(R)$. If $\alpha\notin\reg(RG)$, then since $G$ is torsion-free, we can apply McCoy's Theorem and obtain some non-zero $r\in R$ such that $r\alpha=0$. But then $ra_i=0$ from which we would conclude that $r=0$. Thus, $\alpha\in \reg(RG)$.

\end{proof}

\begin{lemma}\label{cyclic}
Let $R$ be a commutative ring and $G=C_{p^n}=<g>$ where $p$ is a prime and $n\geq 1$.
Then $RG$ has  Property $D$ if and only if $R$ has Property $D$ and $p \in\fN(R)$.
\end{lemma}

\begin{proof}
\noindent $\Rightarrow$ Let $h \in G$ be of order $p$ and $H=<h>$. Set $T = RH$. As mentioned before, both $R$ and $T$ have Property $D$.
If $p\in R$ is not nilpotent, then $p \in T$ is not nilpotent, so $p\in \reg(T)$. As
$$(1+h+\ldots +h_{p-1})(1-h) = 0,$$
it follows that
$1+h+\ldots +h_{p-1} \in \fN(T)$. Hence, for some $k \geq 1$,
$$0 = (1+h+\ldots +h_{p-1})k = p^{k-1}(1+h+\ldots +h_{p-1}).$$

As $p^{k-1}$ is regular, $1+h+\ldots +h_{p-1}= 0$, a contradiction.
\vspace{.1in}

\noindent $\Leftarrow$ We suppose that $R$ has Property $D$ and that $p\in \fN(R)$. We claim that for any $h\in C_{p^n}$, $p-1 + h\in \fN(RC_{p^n})$.

Recall that, since $(p - 1)^{p^n} + 1 \equiv 0 ({\rm mod} p)$, then for each $h\in G$, $(p-1 + h)^{p^n}$ is divisible by $p$ and hence belongs to $ \fN(RG)$. Therefore, so $p-1 + h \in \fN(RC_{p^n})$. Next, let $\alpha\in RC_{p^n}\ssm \fN(RC_{p^n})$ and $\alpha = \sum_{i=0}^{p^n-1} a_ig^i$. But we can also write $\alpha$ as
$$\alpha = -\sum_{0}^{p^n-1} (p-1)a_i + \sum_{0}^{p^n-1} a_i(p-1+g^i)$$
where the right-hand side belongs to $\fN(RC_{p^n}$. Thus, by choice of non-nilpotent $\alpha$, we conclude that
$$-\sum_{0}^{p^n-1} (p-1)a_i$$
is not nilpotent, and thus is regular in $R$, and hence also in $RC_{p^n}$. Therefore, $\alpha$ is the sum of a regular and a nilpotent, whence $\alpha\in \reg(RC_{p^n})$.

\end{proof}

Here is our characterization of when $RG$ has Property $D$.

\begin{theorem}
Let $R$ be a commutative ring and $G$ an abelian group. Then $RG$ has Property $D$ if and only if one of the following holds:

(1) $R$ has Property $D$ and $G$ is torsion-free.

(2) $R$ has Property $D$ and there exists a prime $p$ such that $p\in \fN(R)$ and $t(G)$ is a non-trivial $p$-group.

\end{theorem}

\begin{proof}
$\Leftarrow$ If (1) holds, then $RG$ has Property $D$ by Lemma \ref{tf}.

So, suppose that (2) holds. Let $p$ be a prime such that $t(G)$ is a $p$-group. By Lemma \ref{fg}, it suffices to show that $RH$ has Property $D$ for each finitely generated subgroup $H$ of $G$. Let $H$ be a finitely generated subgroup of $G$. Then $H=K_1K_2$, where $K_1$ is torsion-free and $K_2$ is a torsion group. By (1), we may assume without loss of generality that $K_2\neq \{1_G\}$. Then
$K_1\cong \Z^t$ for some $0\leq t\in \N$, and $K_2\cong C_{p^{k_1}}\times \cdots \times C_{p^{k_r}}$.

We point out that $RK_2\cong (RC_{p^{k_1}})(C_{p^{k_2}}\times \cdots \times C_{p^{k_r}})$. By Lemma \ref{cyclic}, $RC_{p^{k_1}}$ has Property $D$ and $p\in \fN(RC_{p^{k_1}})$. Inductively, we get that $RK_2$ has Property $D$ and that $p\in \fN(RK_2)$. It follows then by Lemma \ref{tf}, that $(RK_2)K_1\cong RH$ has Property $D$.
\vspace{.1in}

$\Rightarrow$ We suppose that $RG$ has Property $D$; every subring of $RG$ has Property $D$. We assume that $G$ is not torsion-free and demonstrate that $t(G)$ is a $p$-group for some prime $p$ and that $p\in \fN(R)$. If it is not this means that there are elements $g,h\in t(G)$ such that $o(g)=p,o(h)=q$ are distinct primes. Then $R<g>$ and $R<h>$ both have Property $D$ and so $p,q\in \fN(R)$ by Lemma \ref{cyclic}. Since $p$ and $q$ are relatively prime, we gather that $1\in\fN(R)$, the desired contradiction. It follows that $t(G)$ is a $p$-group for some prime $p$, and further more $p\in\fN(R)$.

\end{proof}

\end{document}